\documentclass[11pt,oneside,reqno]{article}

\usepackage{amsmath,amsthm,amssymb,color}
\usepackage[authoryear,longnamesfirst]{natbib}

\usepackage{bbm}
\usepackage{mathrsfs}
\usepackage[breaklinks=true]{hyperref}

\numberwithin{equation}{section}
\allowdisplaybreaks[4]

\theoremstyle{plain}
\newtheorem{theorem}{Theorem}[section]

\newtheorem{lemma}[theorem]{Lemma}
\newtheorem{corollary}[theorem]{Corollary}

\theoremstyle{definition}
\newtheorem{definition}{Definition}[section]

\newtheorem{remark}[definition]{Remark}

\makeatletter

\def\phi{\varphi}
\def\^#1{\mathaccent"705E #1}
\def\~#1{\mathaccent"707E #1}

\newcommand{\IZ}{\mathbbm{Z}}
\newcommand{\IR}{\mathbbm{R}}

\def\%#1{\mathcal{#1}}
\newcommand{\law}{\mathscr{L}}

\def\dw{\mathop{d_{\mathrm{W}}}}

\def\dk{\mathop{d_{\mathrm{K}}}}

\newcommand{\lito}{\mathrm{o}}
\newcommand{\hence}{\Rightarrow}

\newcommand{\nsig}{\Sigma\kern-0.5em\raise0.2ex\hbox to 0pt{$\mid$}\kern0.5em}

\newcommand{\tozero}{\to0}

\newcommand{\eps}{\varepsilon}

\newcommand{\I}{\mathrm{I}}

\newcommand{\ahalf}{{\textstyle\frac{1}{2}}}

\newcommand{\eq}{\eqref}

\newcommand{\IE}{\mathbbm{E}}
\newcommand{\IP}{\mathbbm{P}}
\newcommand{\Var}{\mathop{\mathrm{Var}}\nolimits}

\newcommand{\Exp}{\mathop{\mathrm{Exp}}}

\def\be#1\ee{\begin{equation*}#1\end{equation*}}
\def\ben#1\ee{\begin{equation}#1\end{equation}}

\def\bes#1\ee{\begin{equation*}\begin{split}#1\end{split}\end{equation*}}
\def\besn#1\ee{\begin{equation}\begin{split}#1\end{split}\end{equation}}

\def\bg#1\ee{\begin{gather*}#1\end{gather*}}
\def\bgn#1\ee{\begin{gather}#1\end{gather}}

\def\bm#1\ee{\begin{multline*}#1\end{multline*}}
\def\bmn#1\ee{\begin{multline}#1\end{multline}}

\def\ba#1\ee{\begin{align*}#1\end{align*}}
\def\ban#1\ee{\begin{align}#1\end{align}}

\def\klr#1{(#1)}
\def\bklr#1{\bigl(#1\bigr)}
\def\bbklr#1{\Bigl(#1\Bigr)}
\def\bbbklr#1{\biggl(#1\biggr)}

\def\klg#1{\{#1\}}
\def\bklg#1{\bigl\{#1\bigr\}}

\def\abs#1{\vert#1\vert}
\def\babs#1{\bigl\vert#1\bigr\vert}

\def\mid{\vert}

\def\bbbmid{\biggm\vert}

\def\^#1{\ifmmode {\mathaccent"705E #1} \else {\accent94 #1} \fi}
\def\~#1{\ifmmode {\mathaccent"707E #1} \else {\accent"7E #1} \fi}
\def\*#1{#1^\ast}
\edef\-#1{\noexpand\ifmmode {\noexpand\bar{#1}} \noexpand\else \-#1\noexpand\fi}
\def\>#1{\vec{#1}}
\def\.#1{\dot{#1}}

\def\leq{\leqslant}
\def\geq{\geqslant}

\def\atop{\@@atop}

\newcount\minute
\newcount\hour
\newcount\hourMins
\def\now{%
\minute=\time%
\hour=\time \divide \hour by 60%
\hourMins=\hour \multiply\hourMins by 60%
\advance\minute by -\hourMins%
\zeroPadTwo{\the\hour}:\zeroPadTwo{\the\minute}%
}
\def\zeroPadTwo#1{\ifnum #1<10 0\fi#1}

\renewcommand\section{\@startsection {section}{1}{\z@}%
{-3.5ex \@plus -1ex \@minus -.2ex}%
{1.3ex \@plus.2ex}%
{\center\small\sc\mathversion{bold}\MakeUppercase}}

\def\subsection#1{\@startsection {subsection}{2}{0pt}%
{-3.5ex \@plus -1ex \@minus -.2ex}%
{1ex \@plus.2ex}%
{\bf\mathversion{bold}}{#1}}

\def\subsubsection#1{\@startsection{subsubsection}{3}{0pt}%
{\medskipamount}%
{-10pt}%
{\normalsize\itshape}{\kern-2.2ex. #1.}}

\def\blfootnote{\xdef\@thefnmark{}\@footnotetext}

\makeatother

\def\cite{\citet*}

\begin{document}
\title{\bf Exponential approximation for the nearly critical Galton-Watson process and occupation times of Markov chains}
\author{%
Erol Pek\"oz\footnote{School of Management, Boston University, 595 Commonwealth Avenue, Boston, Massachusetts 02215; partially supported by NUS research grant R-155-000-098-133}%
\and
Adrian R\"ollin\footnote{Department of Statistics and Applied Probability, National University of Singapore, 6 Science Drive 2, Singapore 117546; partially supported by NUS research grant R-155-000-098-133}}
\date{}
\maketitle

\noindent\textbf{Abstract}\enskip In this article we provide new applications for exponential
approximation using the framework of \cite{Pekoz2009}, which is based on
Stein's method. We give error bounds for the nearly critical Galton-Watson
process conditioned on non-extinction, and for the occupation times of Markov
chains; for the latter, in particular, we give a new exponential approximation
rate for the number of revisits to the origin for general two dimensional random
walk, also known as the Erd\H{o}s-Taylor theorem.

\bigskip


\bigskip


\bigskip

\noindent Submitted to EJP on August 25, 2010; revised on March 26, 2011; final version accepted on June 30, 2011

\bigskip

\noindent 

\newpage

\section{Introduction}
A new framework for estimating the error of the exponential approximation
was recently developed in \cite{Pekoz2009}, where it was applied to geometric
sums, Markov chain hitting times, and the critical Galton-Watson conditioned on
non-extinction. In this article we provide some generalizations to the
approach of \cite{Pekoz2009} and apply them to study Markov chain occupation
times
and a result of \cite{Erdos1960} for the number of visits to the origin
by the two dimensional  random walk, as well as to get a rate for the
result of \cite{Fahady1971} for the nearly critical Galton-Watson branching
process conditioned on non-extinction.

The main result in \cite{Pekoz2009} that we use is based on Stein's method (see
e.g.
\cite{Ross2007} for an introduction) and can be thought of as formalizing the
intuitive notion that a random
variable $X$ has approximately an exponential distribution if $X$ and $X^e$ are
close in distribution, where $X^e$ has the equilibrium distribution with respect
to $X$ characterized by
\ben									\label{1b}
    \IP[X^e\leq x] = \frac{1}{\IE X}\int_0^x\IP[X>y]dy.
\ee
The equilibrium distribution appears in renewal theory as the time until the
next renewal starting from steady-state.   A renewal process with exponential
inter-renewal times has the exponential distribution for its equilibrium
distribution, and so the above intuition is not surprising.
\cite{Pekoz2009} give bounds on the accuracy of the exponential approximation in
terms of how closely $X$ and $X^e$ can be coupled together on the same
probability space; one version of the result we will use below can be written as
\be
  \sup_{x\geq 0} \babs{\IP[X> x]-e^{-x/\IE X}}\leq 2.46\sqrt{\IE |X-X^e|}.
\ee

Some heuristics for Stein's method can be understood using size-biased random
variables.  For a nonnegative continuous random variable $X$ with probability density
function $f(x)$, the size-biased random variable $X^s$ has  density $xf(x)/\IE
X$.  The size of the renewal interval containing a randomly chosen point as well
as the number of children in the family of a randomly chosen child are examples of
size-biased random variables; see \cite{Brown2006} and \cite{Arratia2010} for
surveys
and applications of size biasing.

Stein's method for the exponential distribution, as well as for some other
nonnegative
distributions, can be viewed in terms of size-biasing.  For the Poisson
approximation to some random variable $X$, the Stein-Chen method (see
\cite{Barbour1992}) gives a bound on the error in terms of how closely $X$ and
$X^s-1$ can be coupled together on the same probability space; these both have
exactly the same distribution when $X$ has a Poisson distribution. For
approximation by a binomial distribution (see \cite{Pekoz2009a}), we can obtain
a bound in terms of how closely $X^s-1$ and $n-(n-X)^s$ can be coupled; both of
these have exactly the same distribution if $X$ is binomial with parameters
$n$ and~$p$.  For the exponential distribution, we can obtain a bound on the
error
in terms of how closely $X$ and $UX^s$ can be coupled, where $U$ is an
independent uniform (0,1) random variable independent of all else;  $X^e$ has
the same distribution as~$UX^s$.  This last approach is the one we use below for
the nearly critical Galton Watson process conditioned on non-extinction.

The organization of this article is as follows. In Section 2 we give the
notation, background and preliminaries.  In Section 3 we consider the setting of
a nearly critical Galton Watson branching process conditioned on non-extinction.
In Section 4 we study general dependent sums, occupation times for Markov chains
and the the number of times the origin is revisited for the two-dimensional
general random walk.

\section{Preliminaries}

We first define the probability metrics we use below. For two
probability distributions $F$ and $G$ define the Kolmogorov metric as
\be
  \dk\klr{F,G} = \sup_{x\in\IR}\babs{F(x)-G(x)}.
\ee
If both distributions have finite expectation, define the Wasserstein
metric
\be
  \dw\klr{F,G} = \int_{\IR}\babs{F(x)-G(x)}dx.
\ee
We can relate the two metrics using
\be
    \dk(P,\Exp(1)) \leq 1.74\sqrt{\dw(P,\Exp(1))};
\ee
see e.g. \cite{Gibbs2002}.

Central to the approach in \cite{Pekoz2009} is the equilibrium distribution from
renewal theory, and we next give the definition we use.

\begin{definition}
Let $X$ be a non-negative random variable with finite mean. We say that a random
variable $X^e$ has the \emph{equilibrium distribution w.r.t.\ $X$} if for all
Lipschitz-continuous $f$
\ben 							    \label{1}
    \IE f(X) - f(0) =\IE X \, \IE f'(X^e).
\ee
\end{definition}
It is straightforward that this implies \eq{1b}. Indeed for nonnegative $X$ having finite first moment, define the distribution
function
\be
    F^e(x) = \frac{1}{\IE X}\int_0^x\IP[X>y]dy
\ee
on $x\geq 0$ and $F^e(x) = 0$ for $x<0$.
Then
\bes
    &\IE f(X)-f(0)
     = \IE \int_0^X f'(s) d s\\
    &\qquad = \IE \int_0^\infty f'(s) I[X>s] ds
     = \int_0^\infty f'(s) \IP[X>s] ds
\ee
so that $F^e$ is the distribution function of $X^e$ and our definition via
\eq{1} is consistent with that from renewal theory.

The size biased distribution will also be used below. We define it as follows.

\begin{definition}
Let $X$ be a non-negative random variable with finite mean. We say that a random
variable $X^s$ has the \emph{size-biased distribution w.r.t.\ $X$} if for all
bounded $f$
\ben                                                            \label{2}
    \IE\bklg{X f(X)}  =\IE X \, \IE f(X^s).
\ee
\end{definition}
It may be helpful in what follows to note that this definition using $f(x)=x^n$ immediately gives $\IE (X^s)^n = \IE X^{n+1}/\IE X$.
We next present the key result from \cite{Pekoz2009} that we will use in the
applications that follow. \begin{theorem}[\cite{Pekoz2009}, Theorem
2.1]\label{thm1} Let $W$ be a non-negative random variable with $\IE W= 1$ and
let $W^e$ have the equilibrium distribution w.r.t.\ $W.$ Then, for
any $\beta>0$,
\be
    \dk\bklr{\law(W),\Exp(1)}\leq 12\beta + 2\IP[\abs{W^e-W}>\beta],
\ee
and, if in addition $W$ has finite second moment,
\ben                                                            \label{3}
    \dw\bklr{\law(W),\Exp(1)}\leq 2\IE\abs{W^e-W}
\ee
\end{theorem}

\section{The nearly critical Galton-Watson branching process}

Consider the Galton-Watson branching process starting
from a single particle in generation  zero, where each particle has an independent and identically distributed number of children according to some distribution having mean~$m$; let $Z_{n}$ be the size of the $n$th
generation.
 For the critical case where $m=1$ and when $\IE Z_1^2 <\infty$ and $\IP (Z_1=0)>0$, it was shown by \cite{Yaglom1947} that the conditional distribution of $Z_n/n$
given $Z_n>0$ converges as $n\rightarrow \infty$ to an exponential distribution.  A
corresponding rate of convergence was first proved with the additional condition $\IE Z_1^3<\infty$ by \cite{Pekoz2009}. In the
super-
and sub-critical cases, respectively when $m>1$ and $m<1$, the limiting distributions are very difficult to
calculate and are only known explicitly in very special cases; see e.g.\
\cite{Bingham1988}. \cite{Fahady1971}, however, were able to show that the
limiting distribution of a
nearly critical branching process conditioned on non-extinction converges to the
exponential distribution as $m\to1$ over general classes of offspring distributions. The
following theorem gives explicit error bounds for the exponential approximation
for any finite $n$ and any $m\neq 1$. To avoid trivial cases, we make the general assumption that
\be
	0<\IP[Z_1=0]<1.
\ee

\begin{theorem}\label{thm2} Consider a Galton-Watson branching process starting
from a single particle at time zero, and let $Z_{n}$ be the size of the $n$th
generation. Assume $\IP[Z_1\geq2]>0$, $m=\IE Z_1 \neq 1$ and
$\IE Z_1^3<\infty$. Let $\alpha = \IP[Z_{1}=1]/\IP[Z_1\geq
2]$,
\ben									\label{4c}
  C=(2+\alpha)^2\bklr{1+\Var Z_1+\IE Z_1^3}^2
\ee
and $\lambda =1/\IE(Z_{n}|Z_{n}>0) = \IP[Z_n>0]/m^n$.  Then
\ben						\label{4}
    \dw\bklr{\law(\lambda Z_{n} | Z_{n} >0),\Exp(1)} \leq C\eta(m,n),
\ee
where
\ben							\label{4b}
  \eta(m,n) = \frac{1-m}{1-m^n}
    +\frac{(1-m)^2}{m(1-m^n)}\sum_{j=1}^{n-1}\frac{m^{2j}}{1-m^{j}}.
\ee
\end{theorem}

It seems difficult to directly deduce rates of convergence from~\eq{4b}. The following estimates are more useful (a proof is given in the Appendix).

\begin{lemma}\label{lem3} For any $n\geq 1$ and $m>1$,
\ben								\label{5}
   \eta(m,n)\leq 2(m-1) + \frac{3+\log(n)}{n},
\ee
and, for any $n\geq 2$ and $\ahalf \leq m < 1$,
\ben								\label{6}
  \eta(m,n)\leq \bklr{4-2\log(1-m)}(1-m) +  \frac{4+2\log(n)}{n}.
\ee
\end{lemma}

\cite{Fahady1971} considered classes $K(a,b)$, $0<a<\infty$, $0<b<\infty$, of offspring distributions such that for all $\law(Z_1)\in K(a,b)$,
\be
	(A) \enskip \IE Z_1^3 \leq a,\qquad
	(B) \enskip \IE Z_1(Z_1-1)\geq b,
\ee
and showed that within each such class the limiting distribution of the conditioned Galton-Watson branching process converges to the exponential as $m\to1$. While retaining $(A)$, it is not too difficult to see that Condition $(B)$ is equivalent to
\be
	(B') \enskip \IP[Z_1\geq 2]\geq b'
\ee
for some $b'>0$ (it is easy to see that $(B')$ implies $(B)$---a proof of the reverse is given in the Appendix). Hence, it is clear that under these assumptions, the constant~$C$ in \eq{4c} will remain bounded as $m\to 1$ and hence Theorem~\ref{thm2} and Lemma~\ref{lem3} give explicit bounds under the conditions of \cite{Fahady1971}. Thanks to our explicit bounds, we can furthermore weaken the assumptions on the offspring distributions in the sense that the third moment of $Z_1$ may grow and $\IP[Z_1\geq 2]\tozero$ as long as $C=\lito\bklr{\frac{1}{m-1}}$ if $m\searrow1$, respectively, $C=\lito\bklr{\frac{-1}{(1-m)\log(1-m)}}$ if $m\nearrow1$.

\begin{proof}[Proof of Theorem~\ref{thm2}] With some modifications, we follow
the line of argument from \cite{Pekoz2009}, which is based on the
size-biased branching tree of \cite{Lyons1995}.

We
assume that the particles in the tree are labeled and ordered. That is, if $w$
and $v$
are two particles in the same generation, then all offspring of $w$ are to
the left of the offspring of $v$, whenever $w$ is to the left of $v$.
We start in
generation $0$ with one particle $v_0$ and let it have a size-biased number of
offspring. Then we pick one of the
offspring of $v_0$ uniformly at random and label it $v_1$. For each of the
siblings (the other offspring from the same parent) of $v_1$ we continue with an independent Galton-Watson branching
process with the original offspring distribution. For $v_1$ we proceed as
we did for $v_0$, i.e., we give it a size-biased number of offspring, pick one
uniformly at random, label it $v_2$, and so on.

Denote by $S_n$ the total number of particles in
generation $n$. Denote by $L_n$ and $R_n$, respectively, the number of particles
to the left (exclusive $v_n$) and to the right (inclusive $v_n$), respectively,
of~$v_n$. Denote by $S_{n,j}$ the number of particles in generation $n$
that stem from any
of the siblings of $v_j$ (but not $v_j$ itself). Likewise, let
$L_{n,j}$ and $R_{n,j}$, respectively, be the number of particles in generation
$n$ that stem from the siblings to the left and right, respectively, of $v_j$.
We have the relations $L_n =
\sum_{j=1}^n L_{n,j}$ and $R_n = 1 + \sum_{j=1}^n R_{n,j}$.

Next let $R_{n,j}'$ be independent random variables such that
\be
    \law(R'_{n,j}) = \law(R_{n,j} | L_{n,j} = 0),
\ee
and, with $A_{n,j} = \{ L_{n,j}=0\}$, define
\be
    R_{n,j}^* = R_{n,j} I_{A_{n,j}} + R_{n,j}' I_{A_{n,j}^c}
    = R_{n,j} + (R_{n,j}' - R_{n,j}) I_{A_{n,j}^c}.
\ee
Define also $R_n^* = 1 + \sum_{j=1}^n R_{n,j}^*$. Below are a few
facts that we will subsequently use to give the proof of the theorem. In what
follows, let $\sigma^2 = \Var Z_1$ and $\gamma=\IE Z_1^3$.
\ba
    (i)&\enskip\rlap{\text{The size-biased distribution of $\law(X)$ is the same
as that of $\law(X | X>0)$;}}\\
    (ii) &\enskip \rlap{\text{$S_n$ has the size-biased distribution of
        $\law(Z_n)$;}}\\
    (iii) &\enskip \rlap{\text{$v_n$ is uniformly distributed
among the particles of generation~$n$;}}\\
    (iv) &\enskip \law(R_n^*) = \law(Z_n | Z_n > 0);
    &(v) &\enskip \IE \klg{R_{n,j}'I_{A_{n,j}^c}}\leq m^{n-j}\sigma^2
\IP[A_{n,j}^c];\\
    (vi) &\enskip \text{$\IE\klg{R_{n,j} I_{A_{n,j}^c}} \leq m^{n-j}\gamma
            \IP[A_{n,j}^c]$;}
    & (vii) &\enskip  \text{ $\IP[A_{n,j}^c]\leq
m^{-1}\sigma^2\IP[Z_{n-j}>0]$.}\\
    (viii)&\enskip {\IP[Z_n>0] \leq (2+\alpha)\frac{m^n(1-m)}{1-m^n}}
\ee
For $(i)$-$(iv)$ see \cite{Pekoz2009}. Using independence,
\be
    \IE \klg{R'_{n,j} I_{A_{n,j}^c}} = \IE R'_{n,j} \IP[A_{n,j}^c]  \leq
    \IE S_{n,j} \IP[A_{n,j}^c]   \leq m^{n-j}\sigma^2 \IP[A_{n,j}^c],
\ee
which proves $(v)$. If $X_j$ denotes the number of
siblings of $v_j$, having the size-biased distribution of $\law(Z_1)$ minus $1$,
we have
\bes
    \IE\klg{R_{n,j} I_{A_{n,j}^c}}
    &\leq m^{n-j}\IE\klg{X_j I_{A_{n,j}^c}}
    \leq m^{n-j}\sum_{k}k\IP[X_j=k,A_{n,j}^c]  \\
    &\leq m^{n-j}\sum_{k}k\IP[X_j=k]\IP[A_{n,j}^c|X_j=k]\\
    &\leq m^{n-j}\sum_{k}k^2\IP[X_j=k]\IP[A_{n,j}^c]\\
    &\leq m^{n-j}\IE X_i^2 \IP[A_{n,j}^c]
    \leq  m^{n-j}\gamma\IP[A_{n,j}^c],
\ee
hence $(vi)$. Now,
\be
    \IP[A_{n,j}^c] = \IE\klg{\IP[A_{n,j}^c|X_j]}
    \leq \IE\klg{X_j\IP[Z_{n-j}>0]} =  m^{-1}\sigma^2\IP[Z_{n-j}>0],
\ee
which proves $(vii)$. Finally, using the Corollary on page 356 of
\cite{Fujimagari1980}, we have
\be
    \IP[Z_n>0] \leq \left(2+\alpha\right) \frac{1-m}{m^{-n}-1},
\ee
which is $(viii)$ (note that the result cited is for bounded offspring
distribution, but easily extends to the unbounded case).

Set $W = \lambda R_n^*$, and note that, due to $(iv)$, $\law(W)=\law(Z_n|Z_n>0)$.
Due to $(i)$ and $(ii)$, $S_n$ has the size-biased distribution with respect to
$R_n^*$. Let $U$ be a uniform random variable on $[0,1]$, independent of all else. Note that, if $Y$ is a random variable, uniformly distributed on the integers $\{1,\dots,n\}$, then $Y-U$ is continuous and uniformly distributed on $[0,n]$. Observing that, given $S_n$, $R_n$ has uniform distribution on $\{1,\dots,S_n\}$ because of $(iii)$, we therefore deduce that $R_n-U$ has uniform distribution on $[0,S_n]$. Hence, $\law(R_n-U)=\law(U S_n)$, which implies that we can set $W^e =\lambda(R_n -
U) $. Applying~\eq{3} and using $(v)$--$(vii)$, we obtain
\bes
\IE\abs{R_n^* - R_n} & \leq \sum_{j=1}^n \IE\klg{R_{n,j}'I_{A_{n,j}^c} +
            R_{n,j}I_{A_{n,j}^c}}\\
     &\leq       \sum_{j=1}^n
    	m^{n-j}\bklr{\sigma^2+\gamma}\IP[A_{n,j}^c] \\
    & \leq \sigma^2+\gamma+ \sum_{j=1}^{n-1}
    	\frac{m^{n-j}\bklr{\sigma^2+\gamma}\sigma^2\IP[Z_{n-j}>0]}{m}\\
        & \leq
    	\sigma^2+\gamma+
\sum_{j=1}^{n-1}
	\frac{m^{n-j}\bklr{\sigma^2+\gamma}\sigma^2(2+\alpha)m^{n-j}(1-m)}{m(1-m^{n-j})}\\
 & \leq
    	\sigma^2+\gamma+\frac{1-m}{m}(2+\alpha)\bklr{\sigma^2+\gamma}\sigma^2 \sum_{j=1}^{n-1} \frac{m^{2j}}{1-m^{j}},
\ee
and, using
\be
	\lambda = \frac{\IP[Z_n>0]}{m^n} \leq \frac{(2+\alpha)(1-m)}{1-m^n},
\ee
we obtain
\be
  \IE\abs{W-W^e}
    \leq
     \lambda/2 + \lambda\IE\abs{R_n^* - R_n} \leq C \eta(m,n),
\ee
which proves~\eq{4}.
\end{proof}

\section{Visits to the origin for a two dimensional simple random walk}

Exponential approximation results for sums of nonnegative random variables $X_1,
X_2,\dots,X_n$ satisfying the condition $\Var(\IE(X_i|X_1,\ldots X_{i-1}))=0$
for all~$i$ were given in \cite[Theorem 3.1]{Pekoz2009}, but not for more
general dependent sums. Here we give a construction of the equilibrium distribution for sums of arbitrarily
dependent nonnegative random variables having finite means, apply it to occupation times for Markov
chains and then illustrate it by getting a new exponential approximation rate
for the
number of times a general irreducible aperiodic two-dimensional integer-valued random walk revisits the origin.

\begin{theorem} \label{thm3}Let $W=\lambda \sum_{i=1}^n X_i$ where $X_1,
X_2,\dots, X_n$ are
(possibly dependent) nonnegative random variables and let $\lambda=
1/\IE\sum_{i=1}^n X_i$.  Suppose, for each $i$ and each $x$, $W_i(x)$ is a random variable such that
\be
  \law(W_i(x)) = \law\bbbklr{\lambda \sum_{m=1}^{i-1} X_m \bbbmid X_{i} =x}.
\ee
For each $i$, let $X_i^s$ be a random variable having the size-biased distribution of~$X_i$. Let
$I$ be independent of all else with $\IP [I=i] = \lambda\IE X_{i}$ and let
$U$ be a uniform random variable on $(0,1)$, independent of all else. Then 
  $W_I(X_I^s) + \lambda U X_I^s$
has the equilibrium distribution with respect to $W$.
In particular, if $X_i \in \{0,1\}$ for all $i$, we have $X^s_i=1$ and hence
$W_I(1)+\lambda U$ has the equilibrium distribution with respect to $W$.
\end{theorem}

\begin{proof} Let $S_m= \lambda\sum_{i=1}^m X_i$. By first conditioning on
$I$ and $U$, and using \eq{2} and $\law(S_i) = \law(W_i(X_i))$ for the third equality, we obtain
\bes
   &\IE  f'\bklr{W_I(X_I^s) + \lambda U X_I^s}\\
   &\qquad=\sum_{i=1}^n\lambda\IE X_i \int_0^1\IE\bklg{f'(W_i(X_i^s)
	+\lambda u X_i^s)}du\\
 &\qquad=\sum_{i=1}^n\lambda\int_0^1\IE \bklg{X_i (f'(W_i(X_i)
	+\lambda u X_i))}du\\
   &\qquad=\sum_{i=1}^n\IE \bklg{ f(W_i(X_i)
      +\lambda X_i)-f(W_i(X_i))} \\
   &\qquad = \sum_{i=1}^n \IE \bklg{f( S_{i})-f(S_{i-1} )} = \IE f(W) - f(0).
  \qedhere
\ee
\end{proof}

\begin{remark}\label{rem}
The argument goes through in the same way when instead we define
\be
  \law(W_i(x)) = \law\bbbklr{\lambda \sum_{m=i+1}^{n} X_m \bbbmid X_{i} =x}.
\ee
\end{remark}

We next apply the above result to Markov chain occupation times.  Our next
result gives a bound on the error of the exponential approximation for the
number of times
a Markov chain revisits its starting state. More general asymptotic results of
this type, but without explicit bounds on the error, go back to
\cite{Darling1957}.

\begin{corollary}\label{col}
Consider a Markov chain started at time zero in a state 0 and let $X_i$ be the
indicator for the event that the Markov chain is in state 0 at time $i$. Let
$W_m=\lambda \sum_{i=1}^m X_i$   and
$\lambda = 1/\IE [\sum_{i=1}^n X_i]$.  Then writing $W\equiv W_n$ we have
\be
  \dw\bklr{\law(W),\Exp(1)}
   \leq 2\lambda + 2\lambda^2 \sum_{i=1}^n \sum_{j=n-i+1}^n \IE X_i\, \IE X_j
\ee
\end{corollary}

\begin{proof}
Using the notation of the previous Theorem \ref{thm3} and Remark \ref{rem}, in
this
setting the
strong Markov property gives
\be
  \law(W_i)
      = \law\bbbklr{\lambda \sum_{m=i+1}^{n} X_m \bbbmid X_{i} =1}
      =\law\bbbklr{\lambda  \sum_{m=1}^{n-i} X_m }
\ee
and so we can let
\be
  W^e = \lambda \sum_{i=1}^{n-I} X_i + \lambda U,
\ee
where $\IP[I=i]=\lambda \IE X_i$, and conditioning on $I$ gives
\be
  \IE |W-W^e|
    \leq \lambda + \lambda^2 \sum_{i=1}^n \sum_{j=n-i+1}^n \IE X_i \IE X_j
\ee
and then~\eq{3} gives the result.
\end{proof}

We next consider a general aperiodic irreducible random walk on the two-dimensional integer lattice
started at the origin.  As a consequence of \cite[p.~24]{Lawler2010} we have the
following lemma.
\begin{lemma}\label{lem2} Let $Z_n$ be
an irreducible and aperiodic random walk on $\IZ^2$ with mean zero and finite
third moment. Then there are positive constants $c_1$ and $c_2$ such that
\be
  \frac{c_1}{n} \leq \IP[Z_n = 0] \leq \frac{c_2}{n}
\ee
for sufficiently large~$n$.
\end{lemma}

We are now able to give a bound on the error of the exponential approximation
for the number
of times the random walk revisits the origin.  This type of result, for simple
random walk, goes back to
\cite{Erdos1960}.

\begin{corollary} Let $Z_n$ be
an irreducible and aperiodic random walk on $\IZ^2$ with mean zero and finite
third moment.
Let $R$ be the number of return visits to the origin by time $n$, and let
$W=\lambda R$, where $\lambda =1/\IE R$. Then, there is constant $C$
independent of $n$ such that
\be
  \dw(\law(W), \law(\Exp(1)) \leq \frac{C}{\log n}.
\ee
for all $n$.
\end{corollary}

\begin{proof}
Let $X_n = \I_{\{Z_n=0\}}$ be the indicator for the event that the random walk
revisits the origin at time $n$.  Lemma~\ref{lem2} gives $\lambda \leq C/\log
n$ and
thus the result follows from Corollary~\ref{col} and, where $C$ may be different
(but independent of $n$) in each instance used,
\be
  \lambda^2 \sum_{i=1}^n \sum_{j=n-i+1}^n \IE X_i \IE X_j
    \leq \frac{C}{(\log n)^2}\sum_{i=1}^n \frac{i}{i(n-i)}
    \leq \frac{C}{\log n}.
\ee
\end{proof}

\begin{remark}
The result for the two-dimensional simple random walk
\be
  \sup_{a<x<b}|\IP[W> x]-e^{-x}| \leq \frac{C\log \log n}{\log n}
\ee
for fixed $a$ and $b$ follows from \cite[Eq.~(3.10)]{Erdos1960}, so the
above corollary can be viewed as a complement and extension. Using the method
of moments, \cite[Theorem~1.1]{Gartner2009} give an argument for the analogous
exponential limit theorem for general random walks, but without a rate of
convergence.
\end{remark}

\section*{Acknowledgments}
The authors are grateful to Andrew Barbour and Rongfeng Sun for stimulating
discussions and for the suggestions (respectively) to study the nearly critical
branching process and the Erd\H{o}s-Taylor theorem. We also thank the anonymous referee for helpful comments.

\begin{appendix}

\section{Proof of Lemma~3.2}
We first need some simple estimates.

\begin{lemma}\label{lem1}
Let $a$, $b$ and $c$ be real numbers, strictly greater than $1$, such that
\ben								\label{9}
  \frac{1}{a} \leq \frac{1}{b} + \frac{1}{c} \leq 1.
\ee
Then
\be
  \frac{\log(a)}{a} \leq \frac{1+\log(b)}{b}+\frac{1+\log(c)}{c}.
\ee
\end{lemma}

\begin{proof}It is clear from the monotonicity of the logarithm function that
for $x,y>0$ we have
\be
   x\log(x)+y\log(y)  \leq (x+y)\log(x+y).
\ee
Hence,
\be
  (x+y)(1-\log(x+y)) \leq x(1-\log(x))+y(1-\log(y)).
\ee
Rewriting this inequality for $x=1/b$ and $y=1/c$, we have
\be
  \frac{1+\log\bklr{\frac{bc}{b+c}}}{\frac{bc}{b+c}}
      \leq\frac{1+\log(b)}{b}+\frac{1+\log(c)}{c}.
\ee
Noting that $\frac{1+\log(a)}{a}$ is a decreasing function for $a\geq 1$ and
noting that $a\geq \frac{bc}{b+c}\geq 1$ from~\eq{9},
\be
  \frac{\log(a)}{a} \leq \frac{1+\log(a)}{a}
  \leq \frac{1+\log\bklr{\frac{bc}{b+c}}}{\frac{bc}{b+c}},
\ee
which proves the claim.
\end{proof}

Let $f$ be a non-negative function on $[a,b]$ for two integers $a$ and $b$. If $f$ is either increasing, decreasing or has exactly one minimum, a simple geometric argument yields that
\ben				\label{6b}
	\sum_{j=a}^{b-1} f(j) \leq f(a)+\int_{a}^b f(x)dx
\ee
(this estimate is not optimal if the function is increasing, but we want to avoid further case distinctions). 

\begin{proof}[Proof of Lemma~\ref{lem3} for $m>1$]
It is straightforward to see that $f(x)=m^{2x}/(m^x-1)$, $x>0$, has exactly one minimum at $x_0=\log(2)/\log(m)$, hence $f(x)$ is decreasing on $0<x\leq x_0$ and increasing on $x\geq x_0$. Using \eq{6b}, we therefore  have
\bes
  \sum_{j=1}^{n-1}\frac{m^{2j}}{m^{j}-1}
  &\leq \frac{m^2}{m-1}
    +\frac{m^{n}-m+\log\bklr{\frac{m^{n}-1}{m-1}}}{\log(m)},
\ee
which implies that
\bes
  \eta(m,n)
    & \leq \frac{m-1}{m^n-1}+\frac{m(m-1)}{m^n-1}
    + \frac{(m-1)^2\bklr{m^{n}-m}}{m(m^n-1)\log(m)}\\
    &\quad + \frac{(m-1)^2 \log\bklr{\frac{m^{n}-1}{m-1}}}{m(m^n-1)\log(m)}
    =: r_1 + r_2 + r_3 + r_4.
\ee
Recall that
\ben							\label{7}
  \frac{m^n -1}{m-1} =\sum_{k=0}^{n-1}m^k\geq n.
\ee
This implies that
\be
  r_1 \leq \frac{1}{n}, \qquad
  r_2  \leq \frac{m}{n} = \frac{m-1}{n}+\frac{1}{n}
    \leq m-1+\frac{1}{n}.
\ee
Furthermore, recalling that $m-1\leq m\log(m)$,
\be
  r_3 \leq \frac{(m-1)\bklr{m^{n}-m}}{m^n-1} \leq m-1.
\ee
Finally,
\bes
  r_4 = \frac{(m-1)^2\log\bklr{\frac{m^{n}-1}{m-1}}}{m(m^n-1)\log(m)}
      \leq \frac{m-1}{m^n-1}\log\bbklr{\frac{m^{n}-1}{m-1}}
      \leq\frac{1+\log(n)}{n}.
\ee
The last estimate is due to the fact that $\log(x)/x$ is
clearly bounded by $(1+\log(x))/x$ for $x>1$, and the latter
is a decreasing function, and then by applying~\eq{7}. Putting the estimates
for $r_1$ through $r_4$ together
proves~\eq{5}.
\end{proof}

\begin{proof}[Proof of Lemma~\ref{lem3} for $m<1$]
Note first that $\frac{y^2}{1-y}$ is increasing on
$0<x<1$, hence $f(x)=m^{2x}/(1-m^x)$ is a decreasing function in $x$. Applying \eq{6b},
\be
  \sum_{j=1}^{n-1}\frac{m^{2j}}{1-m^{j}}
  \leq \frac{m^2}{1-m} 
    +\frac{m-m^{n}+\log\bklr{\frac{1-m}{1-m^{n}}}}{\log(m)},
\ee
which implies that
\bes
  \eta(m,n)
    & \leq \frac{1-m}{1-m^n}+\frac{m(1-m)}{1-m^n}
    + \frac{(1-m)^2\bklr{m-m^{n}}}{m(1-m^n)\log(m)}\\
    &\quad + \frac{(1-m)^2{\log\bklr{\frac{1-m}{1-m^{n}}}}}{m(1-m^n)\log(m)}
    =: r_1 + r_2 + r_3 + r_4.
\ee
As $\frac{1-m^n}{1-m}=\sum_{k=0}^{n-1} m^k\geq m^n n$, we have
\besn							\label{8}
  &\frac{1-m}{m^n}+\frac{1-m}{1-m^n} \leq \frac{1-m}{m^n}+\frac{1}{m^n n} 
  \quad\hence\quad \frac{1-m}{1-m^n}\leq 1-m+\frac{1}{n}.
\ee
Hence
\be
	 r_1+r_2\leq 2(1-m)+\frac{2}{n}.
\ee
It is easy to see that $r_3\leq 0$. Using now that $1-m\leq -2m\log(m)$ for $\ahalf \leq m < 1$,
\be
  r_4
  = \frac{(1-m)^2\log\bklr{\frac{1-m}{1-m^{n}}}}{m(1-m^n)\log(m)}
  \leq \frac{2\log\bklr{\frac{1-m^n}{1-m}}}{\frac{1-m^n}{1-m}}.
\ee
Under the restriction $\ahalf \leq m<1$ and $n\geq 2$, we can now apply
Lemma~\ref{lem1} below for $a=\frac{1-m^n}{1-m}$,
$b=\frac{1}{1-m}$ and $c=n$ due to~\eq{8} and we obtain
\be
  r_4
    \leq 2\bbbklr{\frac{1+\log\bklr{\frac{1}{1-m}}}{\frac{1}{1-m}}
	+\frac{1+\log(n)}{n}}.
\ee
Putting all the estimates together proves~\eq{6}.
\end{proof}

\section{Proof that Condition $(B)$ implies $(B')$}

We will show that $\neg(B')$ implies $\neg(B)$. More specific, under the Condition~$(A)$, we show that $\IP[X\geq 2]\tozero$ implies $\IE X(X-1)\tozero$.

To do this we need to find a vector of probabilities $p_0, p_1, \ldots p_n$ that maximizes
\be
	\sum_{k=0}^n k(k-1)p_k,
\ee
subject to the constraints
\be
	\sum_{k=0}^n k^3 p_k \leq a, \qquad
	\sum_{k=2}^n p_k \leq \eps, \qquad
	p_k\geq 0,\enskip k=0,1,\ldots n, \qquad
	\sum_{k=0}^n p_k = 1.
\ee
 This is a linear programming problem with $n+1$ variables and $n+4$ constraints.  The constraints define a simplex and the fundamental theorem of linear programming tells us the maximum is achieved at a corner point of the simplex where there are $n+1$ binding constraints; this means at most three of the variables $p_k$ can be non-zero at the maximum. As we assume $p_0>0$, we have therefore reduced the problem to just looking at three-point distributions, where one of the three points is at 0.

We consider first the case where neither points are at 1, where we are now trying to find $x,y$ and $p,q$ that maximizes
\be
	x(x-1)p+y(y-1)q
\ee
subject to the constraints
\be
	x^3p+y^3q \leq a, \qquad
	p+q \leq \eps, \qquad
	p,q\geq 0,
	\qquad x,y\geq 2.
\ee
Since the first constraint gives $x\leq (a/p)^{1/3}$ and $y\leq (a/q)^{1/3}$ we have
\be
	x(x-1)p+y(y-1)q \leq a^{2/3}p^{1/3}+a^{2/3}q^{1/3}\leq 2a^{2/3}\eps^{1/3}\rightarrow 0
\ee
as $\eps \rightarrow 0$ and we get our intended result. The case of a point at $y=1$ works out the same way.

\end{appendix}
	
\setlength{\bibsep}{0.5ex}
\def\bibfont{\small}

\end{document}